\documentclass{article}
\usepackage{amssymb, mathrsfs, graphicx, hyperref, fancyhdr}
\usepackage{amsmath}
\usepackage{amsthm}
\usepackage[arc,all]{xy}
\usepackage{tikz}
\usetikzlibrary{shapes,arrows,positioning}
\usepackage{zref-perpage}
\usepackage[nottoc]{tocbibind}

\theoremstyle{definition}
\newtheorem{definition}{Definition}
\newtheorem{example}{Example}
\theoremstyle{theorem}
\newtheorem{theorem}{Theorem}
\newtheorem{lemma}{Lemma}
\newtheorem{proposition}{Proposition}
\newtheorem*{remark}{Remark}
\makeatletter

\title{General $(\alpha, \beta)$ metrics with relatively isotroic mean Landsberg curvature}
\author{ A. Ala, A. Behzadi\footnote{Corresponding author.}  ~~and  M. Rafiei-Rad }

\begin{document}
\maketitle
\begin{abstract}
In this paper, we study a new class of Finsler metrics, $F=\alpha\phi(b^2,s)$, $s:=\beta/\alpha$, defined by a Riemannian metric $\alpha$ and 1-form $\beta$.
It is called general $(\alpha, \beta)$ metric.
In this paper, we assume $\phi$ be coefficient by $s$ and $\beta$ be closed and conformal. 
We find a nessecary and sufficient condition for the metric of relatively isotropic mean Landsberg curvature to be Berwald.
\end{abstract}

\textbf{Keywords:} Finsler geometry, Relatively isotropic mean Landsberg curvature , General ($\alpha$, $\beta$)-metrics.

\section{Introduction}
The $(\alpha, \beta)$ metrics were first introduced by Matsumoto \cite{Matsumoto-1991}.
They are Finsler metrics built from a Riemannian metric $\alpha=\sqrt{a_{ij}y^iy^j}$ and 1-form $\beta=b_i(x)y^i$ and a $C^\infty$ function $\phi(s)$ on a manifold $M$. A Finsler metric of $(\alpha, \beta)$ metrics is given by the form
\begin{equation*}
F:=\alpha\phi(s), \quad s:=\dfrac{\beta}{\alpha}
\end{equation*}
It is known that $F$ is positive and strongly convex on $TM\setminus\{0\}$ if and only if
\begin{equation*}
\phi(s)>0, \quad \phi(s) -s\phi'(s)+(b^2-s^2)\phi''(s)>0, \quad |s|\leq b<b_0,
\end{equation*}
where $b=||\beta||_\alpha$.

The aim of this paper is to study a new class of Finsler metrics given by 
\begin{equation}\label{def-general metric}
F:=\alpha\phi(b^2,s), \quad s:=\dfrac{\beta}{\alpha}
\end{equation}
where $\phi=\phi(b^2,s)$ is a $C^\infty$ positive function and $b=||\beta||_\alpha$ is its norm\cite{Shen-Yuan-2016}. It is called general $(\alpha, \beta)$ metrics.
This kind of metrics is first discussed by Yu and Zhu \cite{Yu-Zhu-2011}.
Many well-known Finsler metrics are general $(\alpha, \beta)$ metrics.
\begin{example}\cite{Shen-Yuan-2016}
The Randers metrics and the square metrics are defined by functions $\phi=\phi(b^2,s)$ in the following form:
\begin{align}
\phi &=\dfrac{\sqrt{1-b^2+s^2}+s}{1-b^2}\\
\phi &=\dfrac{\left( \sqrt{1-b^2+s^2}+s \right)^2}{(1-b^2)^2\sqrt{1-b^2+s^2}}
\end{align}
\end{example}
\begin{example}\cite{Song-Wang-2012}
One Important example of $(\alpha, \beta)$ metric was given by L. Berwald
\begin{equation}
F=\dfrac{\left(\sqrt{(1-|x|^2)|y|^2+\langle x,y\rangle ^2}+\langle x,y\rangle\right)^2} {(1-|x|^2)^2\sqrt{(1-|x|^2)|y|^2+\langle x,y\rangle ^2}}.
\end{equation}
It is a projectively flat Finsler metrics on $\mathbb{B}\subset\mathbb{R}^n$ with flag curvature $K=0$. Berwald's metric can be expressed in form
\begin{equation}
F=\alpha\phi(b^2,s)=\alpha\left(\sqrt{1+b^2}+s^2\right)^2
\end{equation}
where
\begin{align}
&\alpha=\dfrac{\sqrt{(1-|x|^2)|y|^2+\langle x,y\rangle ^2}}{1-|x|^2},\quad
\beta=\dfrac{\langle x,y\rangle}{(1-|x|^2)^{3/2}} \\
&s:=\dfrac{\beta}{\alpha}, \qquad b^2=\dfrac{|x|^2}{1-|x|^2}
\end{align}
\end{example}
\begin{example}\cite{Shen-Yuan-2016}
There is a special class of general $(\alpha, \beta)$ metrics called {\it spherically symmetric metrics}, which are defined on an open subset of $\mathbb{R}^n$ with $\alpha=|y|$ and $\beta=\langle x,y \rangle$,
\begin{equation*}
F=|y|\phi\left( |x|^2, \dfrac{\langle x,y\rangle}{|y|} \right).
\end{equation*}
\end{example}

In Finsler geometry, there are several very important non-Riemannian quantities.
The Cartan torsion $\mathbf{C}$ is a primary quantity. There is another quantity which is determined by the Busemann-Hausdorff volume form, that is the so-called distortion $\tau$. The vertical differential of $\tau$ on each tangent space gives rise to mean Cartan torsion $\mathbf{I}:=\tau_{y^k} dx^k$.
$\mathbf{C}$, $\tau$ and $\mathbf{I}$ are the basic geometric quantities which characterize Riemannian metrics among Finsler metrics. Differentiating $\mathbf{C}$ along geodesics gives rise to the Landsberg curvature $\mathbf{L}$.
The horizontal derivative of $\tau$ along geodesics is the so-called $S$-curvature $\mathbf{S}:=\tau_{|k}y^k$. 
The horizontal derivative of $\mathbf{I}$ along geodesics is called the mean Landsberg curvature $\mathbf{J}:=\mathbf{I}_{|k}y^k$.

By the definition, $\mathbf{J}/\mathbf{I}$ can be regarded as the relative growth rate of the mean Cartan torsion along geodesics. We call a Finsler metric $F$ is of {\it relatively isotropic mean Landsberg curvature} if $F$ satisfies $\mathbf{J}+\tilde{c}F\mathbf{I}=0$, where $\tilde{c}=\tilde{c}(x)$ is a scalar function on the Finsler manifold.
In particular, when $\tilde{c}=0$, Finsler metrics with $\mathbf{J}=0$ are called weak Landsberg metrics \cite{ChernShen-2005}.

We study general $(\alpha, \beta)$ metrics with relatively isotropic mean Landsberg curvature, where $\beta$ is a closed and conformal 1-form, i.e.
\begin{equation}\label{beta-closed-conformal}
b_{i|j}=c(x)a_{ij},
\end{equation}
where $b_{i|j}$ is the covariant derivation of $\beta$ with respect to $\alpha$ and $c=c(x)\neq 0$ is a scalar function on $M$. In \cite{Zohrehvand-Maleki-2016}, Zohrehvand and Maleki
proved that, every Landsberg general $(\alpha, \beta)$ metric is a Berwald metric with condition \eqref{beta-closed-conformal}.
In \cite{MyPaper1}, the authors showed that this result for the metric of mean Landsberg curvature.

In this paper, we prove the following
\begin{theorem}\label{main-theo}
Let $F=\alpha\phi(b^2,s), s:=\beta/\alpha$, be a non-Riemannian general $(\alpha, \beta)$ metric on an $n$-dimensional manifold $M$. Suppose that $\beta$ satisfies \eqref{beta-closed-conformal}. If $\phi=\phi(b^2,s)$ is a polynomial in $b^2$ and $s$, then $F$ is of relatively isotropic mean Landsberg curvature, $\mathbf{J}+\tilde{c}(x) F\mathbf{I}=0$, if and only if it is a Berwald metric.
In this case,
\begin{equation}
\phi(b^2,s)=c_0(b^2)+c_1(b^2)s+c_2(b^2)s^2+\dots+c_m(b^2)s^m,
\end{equation}
where
\begin{equation*}
c_0(b^2)=\dfrac{a_0}{\sqrt{b^2}},\quad
c_1(b^2)=\dfrac{a_1}{b^2},\quad
c_2(b^2)=\dfrac{a_2}{(b^2)^{\frac{3}{2}}},\quad
...,\quad
c_m(b^2)=\dfrac{a_m}{(b^2)^{\frac{m+1}{2}}}
\end{equation*}
and $a_i$, $1\leq i\leq n$ are constants.
\end{theorem}

Because every analytic function can be approximated by a series polynomials, we can assume that $\phi$ is a polynomial in $b^2$ and $s$.

\section{Preliminary}\label{sec-Prelimi.}
Let $F=F(x,y)$ be a Finsler metric on an $n$-dimensional manifold $M$. Let
\[
g_{ij}:=\dfrac{1}{2}[F^2]_{y^iy^j}(x,y)
\]
and $(g^{ij}):=(g_{ij})^{-1}$. For a non-zero vector 
$y=y^i\frac{\partial}{\partial x^i}\mid_x \in T_xM$, $F$ induces an inner product on $T_xM$
\[ g_y(u,v)=g_{ij}u^iv^j , \]
where 
$u=u^i\frac{\partial}{\partial x^i}, v=v^j\frac{\partial}{\partial x^j} \in T_xM$.
$g=\{g_y\}$ is called the fundamental tensor of $F$.

Let
\[ C_{ijk}:=\dfrac{1}{4}[F^2]_{y^iy^jy^k}
=\dfrac{1}{2}\dfrac{\partial g_{ij}}{\partial y^k}. \]
Define symmetric trilinear form $\mathbf{C}:=C_{ijk}(x,y)dx^i \otimes dx^j \otimes dx^k$ on $TM\setminus\{0\}$. We call $\mathbf{C}$ the {\it Cartan torsion}. The {\it mean Cartan torsion} $\mathbf{I}=I_idx^i$ is defined by
\[ I_i:=g^{jk}C_{ijk}. \]
Further, we have (\cite{ChernShen-2005}, \cite{BaoChernShen-2000})
\begin{equation}\label{def-I_i}
I_i=g^{jk} C_{ijk} =\dfrac{\partial}{\partial y^i}\left[\ln\sqrt{\det(g_{jk})} \right]
\end{equation}

For a Finsler metric $F$, the geodesics are characterized locally by a system of 2nd ODEs:
\[ \dfrac{d^2x^i}{dt^2}+ 2G^i\left( x, \dfrac{dx}{dt} \right)=0, \] 
where
\begin{equation}
G^i=\dfrac{1}{4}g^{il}\left\{ [F^2]_{x^my^l}y^m -[F^2]_{x^l} \right\}.
\end{equation}
$G^i$ are called the {\it geodesic coefficients} of $F$.

For a tangent vector $y=y^i\dfrac{\partial}{\partial x^i}\in T_xM$, the {\it Berwald curvature} 
$\mathbf{B}:=B_{jkl}^i dx^j\otimes \dfrac{\partial}{\partial x^i}\otimes dx^k \otimes dx^l$
can be expressed by
\begin{equation}
B_{jkl}^i:=\dfrac{\partial^3 G^i}{\partial y^j \partial y^k \partial y^l}.
\end{equation}
$F$ is a Berwald metric if $\mathbf{B}=0$.
The {\it Landsberg curvature} $\mathbf{L}=L_{ijk}(x,y) dx^i \otimes dx^j \otimes dx^k$ is a horizontal tensor on $TM\setminus\{0\}$ defined by \cite{Shen-2004-book}
\begin{equation}
L_{ijk}:=-\dfrac{1}{2} F F_{y^m}[G^m]_{y^iy^jy^k}.
\end{equation}
$F$ is called a {\it Landsberg metric} if $\mathbf{L}=0$. The {\it mean Landsberg curvature} $\mathbf{J}=J_i dx^i$ is defined by
\begin{equation}
J_i:=g^{jk}L_{ijk}.
\end{equation}
We call $F$ a {\it weak Landsberg metric} if $\mathbf{J}=0$. We say that $F$ is of {\it relatively isotropic mean Landsberg curvature} if $J_i+\tilde{c}(x)F I_i=0$ for a scalar function $\tilde{c}=\tilde{c}(x)$ on $M$.

Now we consider a general $(\alpha,\beta)$ metric:
\begin{definition}
Let $F$ be a Finsler metric on an $n$-dimensional manifold $M$. $F$ is called a {\it general $(\alpha, \beta)$ metric} if it can be expressed as the form \eqref{def-general metric} where $||\beta||_\alpha\leq b_0$ and $\phi=\phi(b^2,s)$ is a positive $C^\infty$ function.
\end{definition}
\begin{proposition}
Let $M$ be an $n$-dimensional manifold. A function $F=\alpha\phi(b^2,s)$ on $TM$ is a Finsler metric on $M$ for any Riemannian metric $\alpha$ and 1-form $\beta$ with $||\beta||_\alpha <b_0$ if and only if $\phi=\phi(b^2,s)$ is a positve $C^\infty$ function satisfying
\begin{equation}\label{condition for positive definite}
\phi>0, \quad \phi -s\phi_2 +(b^2-s^2)\phi_{22}>0,
\end{equation}
where $s$ and $b$ are arbitrary numbers with $|s|\leq b<b_0$.
\end{proposition}
\begin{proof}
It is easy to verify $F$ is a function with regularity and positive homogeneity.
In the following we will verify strong convexity: The $n\times n$ Hessian matrix
\begin{equation*}
(g_{ij}):=(\dfrac{1}{2}[F^2]_{y^iy^j}).
\end{equation*}
For the general $(\alpha, \beta)$ metric $F=\alpha\phi(b^2, \frac{\beta}{\alpha})$, direct computations yield
\begin{align}
[F^2]_{y^i} &= [\alpha^2]_{y^i}\phi^2 +2\alpha^2 \phi\phi_2 s_y^i \\
[F^2]_{y^iy^j} &=[\alpha^2]_{y^iy^j}\phi^2 
+2[\alpha^2]_{y^i}\phi\phi_2 s_{y^j}
+2[\alpha^2]_{y^j}\phi\phi_2 s_{y^i}
+2\alpha^2 [\phi_2]^2 s_{y^i} s_{y^j} \nonumber\\
&\quad +2\alpha \phi_{22} \phi_2 s_{y^i}s_{y^j}
+2\alpha[\phi_2]^2 s_{y^i}s_{y^j}
\end{align}
Direct computations yield
\begin{equation}\label{def-g_ij}
g_{ij}=\rho a_{ij} + \rho_0 b_ib_j +\rho_1(b_i\alpha_{y^i}+b_j\alpha_{y^j})
-s\rho_1 \alpha_{y^i}\alpha_{y^j},
\end{equation}
where
\begin{equation}\label{def-rho}
\rho=\phi(\phi-s\phi_2), \quad
\rho_0=\phi\phi_{22}+(\phi_2)^2,\quad
\rho_1=(\phi-s\phi_2)\phi_2 -s\phi\phi_{22}.
\end{equation}
By Lemma 1.1.1 in \cite{ChernShen-2005}, we find a formula for $\det(g_{ij})$
\begin{equation}\label{def-detg_ij}
\det(g_{ij})=\phi^{n+1}(\phi-s\phi_2)^{n-2}(\phi-s\phi_2 +(b^2-s^2)\phi_{22})\det(a_{ij}).
\end{equation}
Assume that \eqref{condition for positive definite} is satisfied. Then by taking $b=s$ in \eqref{condition for positive definite}, we see that the following inequality holds for any $s$ with
\begin{equation}\label{cond-s=b}
\phi -s \phi_2 >0, \quad |s|<b_0.
\end{equation}
Using \eqref{condition for positive definite}, \eqref{def-detg_ij} and \eqref{cond-s=b}, we get $\det(g_{ij})>0$, namely $(g_{ij}$ is positive-definite. The converse is obvious, so the proof is omitted here.
\end{proof}

By Lemma 1.1.1 in \cite{ChernShen-2005}, we find a formula for $(g^{ij})$
\begin{equation}\label{def-g^ij}
g^{ij}=\rho^{-1}\{ a^{ij}+\eta b^i b^j + \eta_0 \alpha^{1} (b^iy^j+b^jy^i)
+\eta_1 \alpha^{-2} y^iy^j \}
\end{equation}
where $(g^{ij})=(g_{ij})^{-1}$, $(g_{ij})=\frac{1}{2}[F^2]_{y^iy^j}$, $(a^{ij})=(a_{ij})^{-1}$, $b^i=a^{ij} b_j$,
\begin{align}
\eta &=-\dfrac{\phi_{22}}{\phi -s\phi_2 +(b^2-s^2)\phi_{22}},\quad
\eta_0=-\dfrac{(\phi-s\phi_2)\phi_2 -s\phi \phi_{22}}{\phi(\phi -s\phi_2 +(b^2-s^2)\phi_{22})} \label{def-eta}\\
\eta_1 &=\dfrac{(s\phi + (b^2-s^2)\phi_2)((\phi -s\phi_2)\phi_2-s\phi\phi_{22})}{\phi^2(\phi -s\phi_2 +(b^2-s^2)\phi_{22})}\nonumber
\end{align}
\begin{remark}
Note that $\phi_1$ means the derivation of $\phi$ with respect to the first variable $b^2$. In this paper, $\beta$ is closed and conformal 1-form, i.e.
$b_{i|j}=c(x)a_{ij}$. Let \cite{Yu-Zhu-2011}
\begin{align*}
&r_{ij}:=\dfrac{1}{2}(b_{i|j}+b_{j|i}), \quad
r_{00}:=r_{ij}y^iy^j, \quad
r_i:=b^jr_{ji}, \quad
r_0:=r_iy^i,\quad
r^i:=a^{ij}r_j,\quad
r:=b^ir_i,\\
&s_{ij}:=\dfrac{1}{2}(b_{i|j}-b_{j|i}), \quad
s_0^i:=a^{ik}s_{kj}y^j,\quad
s_i:=b^j s_{ji}, \quad
s_0:=s_iy^i,\quad
s^i:=a^{ij}s_j
\end{align*}

For a general $(\alpha, \beta)$ metric, its spray coefficients $G^i$ are related to the spray coefficients $G_\alpha^i$ of $\alpha$ by \cite{Yu-Zhu-2011}
\begin{align}\label{spray coefficient}
G^i =&G_\alpha^i+\alpha Q s_{~0}^i+\lbrace \Theta (-2\alpha Q s_0+r_{00}+2\alpha^2 Rr)+\alpha\Omega (r_0+s_0)\rbrace l^i \nonumber \\
&+\lbrace \Psi (-2\alpha Q s_0+r_{00} +2\alpha^2 Rr)+\alpha \Pi (r_0+s_0) \rbrace b^i -\alpha^2 R(r^i+s^i),
\end{align}
where $l^i:=\dfrac{y^i}{\alpha}$ and
\begin{align*}
Q&:=\dfrac{\phi_2}{\phi-s\phi_2},\quad
R:=\dfrac{\phi_1}{\phi-s\phi_2},\\
\Theta &:=\dfrac{(\phi-s\phi_2)\phi_2-s\phi\phi_{22}}{2\phi(\phi-s\phi_2+(b^2-s^2)\phi_{22})},\quad
\Psi:=\dfrac{\phi_{22}}{2(\phi-s\phi_2+(b^2-s^2)\phi_{22})} \\
\Pi &:=\dfrac{(\phi -s\phi_2)\phi_{12} -s\phi_1 \phi_{22}}{(\phi-s\phi_2)(\phi-s\phi_2+(b^2-s^2)\phi_{22})},\quad
\Omega:=\dfrac{2\phi_1}{\phi}-\dfrac{s\phi+(b^2-s^2)\phi_2}{\phi}\Pi
\end{align*}

When $\beta$ is closed and conformal one-form, i.e. satisfies
\eqref{beta-closed-conformal},
then
\begin{equation*}
r_{00}=c\alpha^2,~~~ r_0=c\beta,~~~ r=cb^2,~~~ r^i=cb^i,~~~ s_{~0}^i=s_0=s^i=0
\end{equation*}
Substituting this into \eqref{spray coefficient} yields \cite{Yu-Zhu-2011}
\begin{equation}\label{subs-spray-coeff}
G^i:=G_\alpha^i +c\alpha^2\lbrace \Theta (1+2Rb^2)+s\Omega \rbrace l^i+ c\alpha^2\lbrace \Psi(1+2Rb^2)+s\Pi -R \rbrace b^i
\end{equation}
If we have
\begin{align}
E &:=\dfrac{\phi_2 +2s\phi_1}{2\phi}-H\dfrac{s\phi +(b^2-s^2)\phi_2}{\phi} \label{def-E} \\
H&:=\dfrac{\phi_{22}-2(\phi_1-s\phi_{12})}{2(\phi-s\phi_2+(b^2-s^2)\phi_{22})}, \label{def-H}
\end{align}
then from \eqref{subs-spray-coeff}
\begin{equation}\label{spray-coeff with cond-closed and conf}
G^i:=G_\alpha^i+c\alpha^2 E l^i+c\alpha^2 Hb^i.
\end{equation}
\end{remark}

\begin{proposition}\label{prop-meanLandsberg-prop}\cite{MyPaper1}
Let $F=\alpha\phi(b^2,s),~s=\beta/\alpha$, be a general $(\alpha, \beta)$-metric on an $n$-dimensional manifold $M$. Suppose that $\beta$ satisfies \eqref{beta-closed-conformal}, then the weak Landsberg curvature of $F$ is given by
\begin{equation}\label{prop-mean Landsberg}
J_j=-\dfrac{c\phi}{2\rho} W_{j},
\end{equation}
where
\begin{align}\label{def-W_j}
W_j:=& \big\{ (E-sE_2)(n+1)\phi_2 +3E_{22} \phi_2(b^2-s^2)-sE_{22}(n+1) \phi+E_{222}\phi(b^2-s^2)  \nonumber\\
&\quad +\lbrace (H_2-sH_{22})(n+1)+H_{222}(b^2-s^2)\rbrace (s\phi +(b^2-s^2)\phi_2)  \nonumber\\
&\quad +3\eta (E-sE_2)\phi_2(b^2-s^2) +3\eta E_{22}\phi_2(b^2-s^2)^2 -3s\eta E_{22}\phi(b^2-s^2) \nonumber\\
&\quad+ \eta E_{222}(b^2-s^2)^2 \phi +\eta [ 3(H_2-sH_{22})(b^2- s^2)+H_{222}(b^2-s^2)^2 ]\nonumber\\
&\quad \times (s\phi +(b^2-s^2)\phi_2) \big\} (b_j-sl_j)
\end{align}
where $\rho$ and $\eta$ is defined in \eqref{def-rho} and \eqref{def-eta} and $l_j:=a_{ij}l^i$.
\end{proposition}

\begin{proposition}\label{conditon-mean-Landsberg}\cite{MyPaper1}
Let $F=\alpha\phi(b^2,s),~s=\beta/\alpha$, be a general $(\alpha, \beta)$-metric on an $n$-dimensional manifold $M$. Suppose that $\beta$ satisfies \eqref{beta-closed-conformal}, then $F$ is weak Landsberg metric if and only if the following equations hold:
\begin{align}
&E_{22}=0 ,~~ H_{222}=0, \label{equation of weak Landsberg-1}\\
&(E-sE_2)\phi_2 +(H_2-sH_{22})(s\phi +(b^2-s^2)\phi_2 =0 \label{equation of weak Landsberg-2}
\end{align}
\end{proposition}

\begin{theorem}\label{Main Theo in My paper 1} \cite{MyPaper1}
Let $F=\alpha\phi(b^2,\frac{\beta}{\alpha})$ be a non-Riemannian general $(\alpha,\beta)$-metric on an $n$-dimentional manifold $M$ and $\beta$ satisfies \eqref{beta-closed-conformal}. Then $F$ is a weak Landsberg metric if and only if it is Landsberg metric.
\end{theorem}

\section{Proof of Theorem \ref{main-theo}}
In this section, we  prove Theorem \ref{main-theo}. 
From \eqref{def-I_i} and \eqref{def-g_ij}, we have
\begin{align}
I_j&=\dfrac{\partial}{\partial y^j}\left[ \ln\sqrt{\det(g_{kl})}\right] \nonumber\\
&=\dfrac{1}{2\alpha}
\Big\{ (n+1)\dfrac{\phi_2}{\phi} -(n-2)\dfrac{s\phi_{22}}{\phi-s\phi_2}+\dfrac{(b^2-s^2)\phi_{222}-3s\phi_{22}}{\phi-s\phi_2+(b^2-s^2) \phi_{22}} \Big\} (b_j-sl_j) \nonumber\\
&=\dfrac{1}{2\alpha \rho} \Big\{
\dfrac{(b^2-s^2)(\phi-s\phi_2) \phi\phi_{222}+(n+1)(\phi-s\phi_2)^2 \phi}{\phi-s\phi_2+(b^2-s^2) \phi_{22}} \nonumber\\
&\quad -(n-2)(b^2-s^2)s\phi\phi_{22}\eta  +(n+1)(\phi-s\phi_2)[(b^2-s^2)\phi_2 -s\phi]\eta  \Big\} (b_j-sl_j).\label{I_jproof}
\end{align}
We must mention the following lemmas firstly.
\begin{lemma}
Let $F=\alpha\phi(b^2,s),~s=\beta/\alpha$, be a general $(\alpha, \beta)$-metric on an $n$-dimensional manifold $M$ and $\beta$ satisfies \eqref{beta-closed-conformal}. Then $F$ is of relatively isotropic mean Landsberg curvature if and only if $\phi$ satisfies the following ODE:
\begin{equation}\label{eq-lemma1}
\dfrac{\phi}{2 \rho}\left(c W_j  +\tilde{c}  V_j\right)=0,
\end{equation}
where
$W_j$ is defined in \eqref{def-W_j},
and
\begin{align*}
V_j&:=\dfrac{(b^2-s^2)(\phi-s\phi_2) \phi\phi_{222}+(n+1)(\phi-s\phi_2)^2 \phi}{\phi-s\phi_2+(b^2-s^2) \phi_{22}} \nonumber\\
&\quad -(n-2)(b^2-s^2)s\phi\phi_{22}\eta  +(n+1)(\phi-s\phi_2)[(b^2-s^2)\phi_2 -s\phi]\eta
\end{align*}
\end{lemma}
\begin{proof}
By Proposition \ref{prop-meanLandsberg-prop}, 
\begin{align}
J_j+\tilde{c}F I_j &=\dfrac{c\phi}{2\rho}
\Big\{ [1+n+3(b^2-s^2)\eta ][(E-sE_2)\phi_2 +(H_2-sH_{22})(s\phi +(b^2-s^2)\phi_2)]  \nonumber\\
&\quad +(b^2-s^2)[1+(b^2-s^2)\eta][s\phi+(b^2-s^2)\phi_2 ] H_{222} \nonumber\\
&\quad +\lbrace 3(b^2-s^2)[1+(b^2-s^2)\eta]\phi_2-[1+n+3(b^2-s^2)\eta]s \phi\rbrace E_{22} \nonumber\\
&\quad +(b^2-s^2)[1+(b^2-s^2)\eta]\phi E_{222} \Big\} (b_j-sl_j)  \nonumber\\
&\quad + \dfrac{\tilde{c}\phi}{2 \rho} \Big\{
\dfrac{(b^2-s^2)(\phi-s\phi_2) \phi\phi_{222}+(n+1)(\phi-s\phi_2)^2 \phi}{\phi-s\phi_2+(b^2-s^2) \phi_{22}} +(n-2)(b^2-s^2)\nonumber\\
&\quad \times s\phi\phi_{22}\eta  -(n+1)(\phi-s\phi_2)[(b^2-s^2)\phi_2 -s\phi]\eta  \Big\} (b_j-sl_j) .
\end{align}
\end{proof}

By use of Maple program, we can immediately get the following lemma.
\begin{lemma}
Let $NJFI$ denote the numerator of the left of \eqref{eq-lemma1}, then \eqref{eq-lemma1} holds if and only if
\begin{equation}\label{NJFI}
NJFI=0.
\end{equation}
Also, Let $NE_{22}$, $NH_{222}$ and $NP$ denote the numerators of $E_{22}$, $H_{222}$ and Eq. \eqref{equation of weak Landsberg-2}, respectively. Then from \eqref{equation of weak Landsberg-1} and \eqref{equation of weak Landsberg-2}, we have
\begin{align}
& NE_{22}=0,\quad NH_{222}=0 \label{equation of weak Landsberg-12}\\
& NP=0, \label{equation of weak Landsberg-22}
\end{align}
\end{lemma}

By assumption, $F$ is of relatively isotropic mean Landsberg curvature. Express $\phi(b^2,s)$ as below.
\begin{equation}\label{phi-poly}
\phi(b^2, s)=c_0(b^2)+c_1(b^2)s+c_2(b^2)s^2+ \dots +c_m(b^2)s^m,\quad m\geq 1.
\end{equation}

Plugging \eqref{phi-poly} to $NJFI$ yields a polynomial in $s$. Denote the order of $NJFI$ by $r$. Then \eqref{NJFI} can be rewritten as follows.
\begin{equation}
v_i s^i h_j= 0 , \quad 0\leq i \leq r,
\end{equation}
where $v_i~ (0\leq i \leq r)$  are independent of $s$.

By using Maple program, we can get following results:

\textbf{Case 1.} $m=1$: $\phi(b^2,s)=c_0(b^2)+c_1(b^2)s$, where $c_1(b^2)\neq 0$. We can get $r=2$ and
\begin{equation}
v_2:=(n+1)\left\{ 
2 ~c ~c_1(b^2) \left[ 2 c_1(b^2) c_0'(b^2) - c_0(b^2) c_1'(b^2)  \right]
+\tilde{c} c_1^3(b^2)
 \right\},
\end{equation}
In this case, because $v_2=0$, so $\tilde{c}$ must be zero.

In here, we can obtain the form of $c_0(b^2)$ and $c_1(b^2)$. Plugging the $\phi$ into \eqref{equation of weak Landsberg-12} and \eqref{equation of weak Landsberg-22} yields $NH_{222}=0$, and
\begin{align}
& c_1^3(b^2) \left\{2b^2 c_0'(b^2) +c_0(b^2)  \right\}
+2c_0^2(b^2) \left\{ c_0(b^2) c_1'(b^2) - 2 c_1(b^2) c_0'(b^2)  \right\}=0, \label{NE_22 - phi - case1} \\
&\quad \nonumber \\
&2c_0^2(b^2) \left\{ -c_0(b^2) c_1'(b^2) + 2 c_1(b^2) c_0'(b^2)  \right\} c_1(b^2)s^2   \nonumber   \\
&\quad +2 c_1^3(b^2) \left\{2b^2 c_0'(b^2) +c_0(b^2)  \right\} s
+c_0(b^2)c_1^2(b^2)\left\{2b^2 c_0'(b^2) +c_0(b^2)  \right\}=0 ,
\label{NP-phi- case1}
\end{align}
From \eqref{NE_22 - phi - case1} and \eqref{NP-phi- case1}, we obtain the following ODE:
\begin{align}
& c_0(b^2) c_1'(b^2) - 2 c_1(b^2) c_0'(b^2) = 0, \\
& 2b^2 c_0'(b^2) +c_0(b^2)=0,
\end{align}
By solving the above ODE, we have 
\begin{equation}\label{solution of case1 }
c_0(b^2)=\frac{a_0}{\sqrt{b^2}}, \quad c_1(b^2)=\dfrac{a_1}{b^2}.
\end{equation}

\textbf{Case 2.} $m=2$. $\phi(b^2,s)=c_0(b^2)+c_1(b^2)s+c_2(b^2)s^2$, where $c_2(b^2)\neq 0$. By using Maple, We can get $r=17$ and
\begin{equation}
v_{17}=c f_{17c} + \tilde{c} f_{17\tilde{c}},
\end{equation}
where $f_{17c}$ is independent of $s$ and
\begin{equation}
 f_{17\tilde{c}}:=927 n \tilde{c} c_2^9(b^2).
\end{equation}
In this case, because $v_{17}=0$, so $\tilde{c} $ must be zero.

We can obtain the form of $c_0(b^2)$, $c_1(b^2)$ and $c_2(b^2)$ too.
Plugging the $\phi$ into \eqref{equation of weak Landsberg-12} and \eqref{equation of weak Landsberg-22} and similar argument yields the following ODE:
\begin{align}
& 2b^2 c_2(b^2) + c_0(b^2)=0, \\
& 2 c_1(b^2) c_2'(b^2) - 3 c_2(b^2) c_1'(b^2)=0, \\
& c_2(b^2)\left\{ 2b^2 c_2'(b^2) + 3 c_2(b^2) - 3c_0'(b^2)  \right\}
+ c_0(b^2) c_2'(b^2)=0,
\end{align}
Then
\begin{equation}
c_0(b^2)=\dfrac{a_0}{\sqrt{b^2}},\quad c_1(b^2)=\dfrac{a_1}{b^2},\quad c_2(b^2)=\dfrac{a_2}{(b^2)^{\frac{3}{2}}}.
\end{equation}

It is not hard to prove by induction that given any $m\geq 3$ in \eqref{phi-poly}, the function $\tilde{c}$ must vanish.

In sum, we have proved that, if $F=\alpha \phi (b^2, s)$ is of relatively isotropic mean Landsberg curvature and $\phi$ be polynomial in $b^2$ and $s$, then $F$  must be a weak Landsberg metric. Then $F$ is a Berwald metric by Theorem \ref{Main Theo in My paper 1}.

In this case, if
\begin{equation*}
\phi(b^2,s)=c_0(b^2)+c_1(b^2)s+c_2(b^2)s^2+\dots+c_m(b^2)s^m,
\end{equation*}
then
\begin{equation*}
c_0(b^2)=\dfrac{a_0}{\sqrt{b^2}},\quad
c_1(b^2)=\dfrac{a_1}{b^2},\quad
c_2(b^2)=\dfrac{a_2}{(b^2)^{\frac{3}{2}}},\quad
...,\quad
c_m(b^2)=\dfrac{a_m}{(b^2)^{\frac{m+1}{2}}},
\end{equation*}
and $a_i$, $1\leq i\leq n$ are constants.

\appendix

\end{document}